\documentclass[11pt,reqno]{amsart}
\usepackage{amsmath}
\usepackage{amsthm}
\usepackage{latexsym}
\usepackage{graphicx}
\usepackage{amssymb}
\usepackage[latin1]{inputenc}

\theoremstyle{plain}

\numberwithin{equation}{section}
\newtheorem{thm}{Theorem}[section]

\newtheorem{lemma}[thm]{Lemma}
\newtheorem{definition}[thm]{Definition}
\newtheorem{proposition}[thm]{Proposition}
\newtheorem{corollary}[thm]{Corollary}
\newtheorem{remark}[thm]{Remark}
\begin{document}
\setcounter{page}{1}
\title[Accelerations of generalized Fibonacci sequences]{Accelerations of generalized Fibonacci sequences}
\author{Marco Abrate}
\address{Dipartimento di Matematica\\
                Universit\`{a} di Torino\\
                Via Carlo Alberto 8\\
                Torino, Italy}
\email{marco.abrate@unito.it}
\author{Stefano Barbero}
\address{Dipartimento di Matematica\\
                Universit\`{a} di Torino\\
                Via Carlo Alberto 8\\
                Torino, Italy}
\email{stefano.barbero@unito.it}
\author{Umberto Cerruti}
\address{Dipartimento di Matematica\\
                Universit\`{a} di Torino\\
                Via Carlo Alberto 8\\
                Torino, Italy}
\email{umberto.cerruti@unito.it}
 \author{Nadir Murru}
\address{Dipartimento di Matematica\\
                Universit\`{a} di Torino\\
                Via Carlo Alberto 8\\
                Torino, Italy}
\email{nadir.murru@unito.it}
\begin{abstract}
In this paper we study how to accelerate the convergence of the ratios $(x_n)$ of generalized Fibonacci sequences. In particular, we provide recurrent formulas in order to generate subsequences $(x_{g_n})$ for every linear recurrent sequence $(g_n)$ of order 2. Using these formulas we prove that some approximation methods, as secant, Newton, Halley and Householder methods, can generate subsequences of $(x_n)$. Moreover, interesting
properties on Fibonacci numbers arise as an application. Finally, we apply all the results to the convergents of a particular continued fraction which represents quadratic irrationalities.
\end{abstract}
\maketitle
\section{Introduction}
The problem of finding sequences of approximations for a solution of a certain equation $f(t)=0$ is really intriguing. There are various and famous methods to generate such sequences in literature, as Newton, secant and Halley methods.
But another interesting problem is how to increase the rate of convergence for the different methods. One of the answer consists in selecting some subsequences of the starting approximations sequences which accelerate the convergence process. In the particular case when $f(t)$ corresponds to a second degree polynomial, we may think about it as a characteristic polynomial of a generalized Fibonacci sequence. The ratio of its consecutive terms generates a new sequence  which converges to the root of larger modulus of $f(t)$, when it exists and it is real. Many authors have given methods to accelerate the rate of convergence of this sequence. McCabe and Phillips \cite{Aitken} have studied how some methods of approximations, like Newton and secant methods, provide an acceleration of such sequence for suitable initial values. Gill and Miller \cite{Gill} have found similar results for the Newton method, only for the ratios of consecutive Fibonacci numbers. Many other works have been developed about this argument. In \cite{Jam} and \cite{Muskat}, e.g., the results of McCabe and Phillips \cite{Aitken} have been generalized for ratios of non--consecutive elements of generalized Fibonacci sequences. In \cite{Nor} these results have been extended to other approximation methods, like the Halley method. The aim of this paper is to give new interesting points of view to prove in an easier way some known results about accelerations of generalized Fibonacci sequences, providing recurrent formulas in order to generate these subsequences. The accelerations related to approximation methods become particular cases of our work. In the first section we will expose our approach, showing in the second section  the relationship with known approximation methods as Newton, secant, Halley and Householder methods. Moreover our results, in addition to a straightforward proof of some identities about Fibonacci numbers, can be applied to convergents of a particular continued fraction representing quadratic irrationalities. We will prove in the last section
that particular subsequences of these convergents correspond to Newton, Halley and secant approximations for the root of larger modulus of a second degree polynomial, when it exists and it is real.

\section{Accelerations of linear recurrent sequences of the second order}
\begin{definition}
We define, over an integral domain with unit $R$, the set $\mathcal{S}(R)$ of sequences $a=(a_n)_{n=0}^{+\infty}$, and the set $\mathcal{W}(R)$ of linear recurrent sequences. Furthermore, we indicate with $a=(a_n)_{n=0}^{+\infty}=\mathcal W(\alpha,\beta,p ,q)$ the linear recurrent sequence of order 2 with characteristic polynomial $t^2-pt+q$ and initial conditions $\alpha$ and $\beta$, i.e.,
$$\begin{cases} a_0=\alpha \cr a_1=\beta \cr a_n=pa_{n-1}-qa_{n-2}\quad \forall n\geq2\ .  \end{cases}$$
\end{definition}
If we consider the sequence $U=\mathcal W(0,1,p,q)$, it is well--known that the sequence $x=(x_n)_{n=2}^{+\infty}$ defined by
\begin{equation}\label{eq:x_n}x_n=\cfrac{U_{n}}{U_{n-1}},\quad \forall n \geq 2\end{equation}
converges to the root of larger modulus (when it is real and it exists) of the characteristic polynomial of $U$. Thus, every subsequence $(x_{g_n})_{n=0}^{+\infty}$ (for some sequence $(g_n)_{n=0}^{+\infty}$) can be considered as a convergence acceleration of the starting sequence $x$. Cerruti and Vaccarino \cite{CerrVacc}  have deeply studied the applications of the companion matrix of linear recurrent sequences. Here we consider some of these results which are useful to find the explicit formulas for the acceleration of the sequences $U$ and $x$. We define
$$M=\begin{pmatrix} 0 & 1 \cr -q & p \end{pmatrix},$$
as the companion matrix of a linear recurrent sequence with characteristic polynomial $t^2-pt+q$.
We recall that some authors may use the transpose of $M$ as the companion matrix.
\begin{lemma}
The sequences $U=\mathcal W(0,1,p,q)$ and $T=\mathcal W(1,0,p,q)$ satisfy the following relations
\begin{equation}\label{UT}\begin{cases} T_{n+1}=-qU_n \cr U_{n+1}=T_n+pU_n  \end{cases},\quad \forall n\geq0.\end{equation} 
\begin{equation}\label{Mpow} M^n=\begin{pmatrix} T_n & U_n \cr T_{n+1} & U_{n+1}  \end{pmatrix}=\begin{pmatrix} T_n & U_n \cr -qU_n & T_n+pU_n  \end{pmatrix},\end{equation}
where $M$ is the companion matrix of $U$ and $T$.
\end{lemma}
\begin{proof}
Using the Binet formula and initial conditions, we easily check that
$$U_n=\cfrac{\alpha^n-\beta^n}{\alpha-\beta},\quad T_n=\cfrac{\alpha\beta^n-\beta\alpha^n}{\alpha-\beta},\quad \forall n\geq0,$$
where $\alpha$ and $\beta$ are the zeros of the characteristic polynomial $t^2-pt+q$. Thus, remembering that $\alpha\beta=q$, we immediately obtain (\ref{UT}) .
In order to prove (\ref{Mpow}), we observe that the characteristic polynomial of $U$ and $T$ annihilates $M$, i.e., we have 
$$M^0=\begin{pmatrix} 1 & 0 \cr 0 & 1 \end{pmatrix},\quad M=\begin{pmatrix} 0 & 1 \cr -q & p \end{pmatrix},\quad M^n=pM^{n-1}-qM^{n-2},\quad \forall n\geq2.$$
So, using (\ref{UT}), we clearly find that the entries of $M^n$ in the first column are the terms $T_n$, $T_{n+1}$ of $T$, and in the second column are the terms $U_n$, $U_{n+1}$ of $U$.   
\end{proof}

\noindent Dealing with definition (\ref{eq:x_n}) of $x$, the ratio of two consecutive terms of any  sequence $a \in \mathcal{W}(R)$ has a closed formula, as we show in the following
\begin{thm}\label{teo:y_n} Let us consider any sequence $a=\mathcal{W}(a_0,a_1,p,q)$ and the sequence 
$y =(y_n)_{n=2}^{+\infty}$, with $y_n=\cfrac{a_n}{a_{n-1}}$, then
\begin{equation}\label{ratio}
y_n=\cfrac{a_n}{a_{n-1}}=\cfrac{a_1x_n-a_0q}{a_0x_n+a_1-a_0p}\quad \forall n\geq2\ ,
\end{equation}
where $x_n$ is the $n$--th term of the sequence $x$ defined by (\ref{eq:x_n}).
\end{thm}
\begin{proof} By direct calculation we have
\begin{eqnarray*}
y_n&=&\cfrac{a_n}{a_{n-1}}=\cfrac{a_1U_n+a_0T_n}{a_1U_{n-1}+a_0T_{n-1}}=\cfrac{a_1U_n+a_0T_n}{a_1U_{n-1}+a_0\left(pU_{n-1}-qU_{n-2}-pU_{n-1}\right)}\\
   &=&\cfrac{a_1U_n-a_0qU_{n-1}}{a_1U_{n-1}+a_0U_{n}-a_0pU_{n-1}}=\cfrac{a_1\cfrac{U_n}{U_{n-1}}-a_0q}{a_0\cfrac{U_n}{U_{n-1}}+a_1-a_0p}=\\
   &=&\cfrac{a_1x_n-a_0q}{a_0x_n+a_1-a_0p}\ .
\end{eqnarray*}
\end{proof}

\noindent Now, as a very important consequence of the previous theorem, we can evaluate the terms of the sequence $x$ shifted by $m$ positions

\begin{corollary}\label{cor:xshift} 
\noindent If $x$ is the sequence defined by (\ref{eq:x_n}), then for all $m\geq2-n$ 
$$
x_{n+m}=\cfrac{x_{m+1}x_n-q}{x_n+x_{m+1}-p}\ .
$$
\end{corollary}

\begin{proof} Let us consider $a=\mathcal{W}(U_m,U_{m+1},p,q)$, that is $a_n=U_{n+m}$ for all $n\geq0$. So the sequence $y$, introduced in the previous theorem, becomes  $y=(y_n=x_{n+m})_{n=2}^{+\infty}$, and, from (\ref{ratio}), we obtain
\begin{equation}
x_{n+m}=\cfrac{U_{m+1}x_n-U_mq}{U_mx_n+U_{m+1}-U_mp}=\cfrac{x_{m+1}x_n-q}{x_n+x_{m+1}-p}\ .
\end{equation} 
\end{proof}

\noindent The basis sequences $U$ and $T$ play a central role in the following results. The accelerations for such sequences yield explicit formulas for the acceleration of the convergent sequence $x=(x_n)_{n=2}^{+\infty}$. Moreover, some interesting properties on Fibonacci numbers arise as an application.
\noindent We recall without proof an important theorem about the generalized Fibonacci sequences.
\begin{thm}[\cite{CerrVacc}, Th 3.1 and Cor. 3.3] \label{teo:cerr} Denote by $V=\mathcal{W}(2,p,p,q)$ the Lucas sequence with parameters $p$ and $q$, then
$$
W_{mn}(0,1,p,q)=W_{m}(0,1,p,q)\cdot W_{n}(0,1,V_m,q^m),
$$
$$
W_{mn}(1,0,p,q)=W_n(1,0,V_m,q^m)+W_m(1,0,p,q)W_n(0,1,V_m,q^m).
$$ 
\end{thm}
\begin{thm}\label{teoUT}
Given the sequences $U=\mathcal{W}(0,1,p,q)$, $T=\mathcal{W}(1,0,p,q)$, and $x$ as defined in equation (\ref{eq:x_n}), if we choose a sequence $g=\mathcal W(i,j,s,t)$, for some $i$, $j\geq2$, then
$$
U_{g_n}=a_2U^{(s)}_{g_{n-1}}T^{(-t)}_{g_{n-2}}+b_2T^{(s)}_{g_{n-1}}U^{(-t)}_{g_{n-2}}+(a_1b_2+a_2b_4)U^{(s)}_{g_{n-1}}U^{(-t)}_{g_{n-2}}$$
$$T_{g_n}=T^{(s)}_{g_{n-1}}T^{(-t)}_{g_{n-2}}+a_1U^{(s)}_{g_{n-1}}T^{(-t)}_{g_{n-2}}+b_1T^{(s)}_{g_{n-1}}U^{(-t)}_{g_{n-2}}+(a_1b_1+a_2b_3)U^{(s)}_{g_{n-1}}U^{(-t)}_{g_{n-2}}\ ,$$
for every $n\geq2$, where $U^{(s)}=\mathcal{W}(0,1,V_s,q^s)$, $ T^{(s)}=\mathcal{W}(1,0,V_s,q^s)\,$ and $$ M^s=\begin{pmatrix} a_1 & a_2 \cr a_3 & a_4 \end{pmatrix}, \quad M^{-t}=\begin{pmatrix} b_1 & b_2 \cr b_3 & b_4 \end{pmatrix}\ . $$
Moreover, if we choose $x_i$ and $x_j$, for some $i$, $j\geq2$, as the initial steps of the acceleration of $x$, then
\begin{equation} \label{x} x_{g_n}=\cfrac{q^2a_2x_{g_{n-1}}^{(s)}+q^2b_2x_{g_{n-2}}^{(-t)}-q(a_1b_2+a_2b_4)x_{g_{n-1}}^{(s)}x_{g_{n-2}}^{(-t)}}{q^2-qa_1x_{g_{n-1}}^{(s)}-qb_1x_{g_ {n-2}}^{(-t)}+(a_1b_1+a_2b_3)x_{g_{n-1}}^{(s)}x_{g_{n-2}}^{(-t)}}\ ,\quad \forall n\geq2 \ , \end{equation}
where $x^{(s)}=-q\cfrac{U^{(s)}}{T^{(s)}}\ .$ 
\end{thm}
\begin{proof} We know as a consequence of Theorem \ref{teo:cerr} that 
$$ M^{sg_{n-1}}=\begin{pmatrix} T^{(s)}_{g_{n-1}}+a_1U^{(s)}_{g_{n-1}} & a_2U^{(s)}_{g_{n-1}} \cr a_3U^{(s)}_{g_{n-1}} & T^{(s)}_{g_{n-1}}+a_4U^{(s)}_{g_{n-1}} \end{pmatrix}\ ,$$ 
$$ M^{-tg_{n-2}}=\begin{pmatrix} T^{(-t)}_{g_{n-2}}+b_1U^{(-t)}_{g_{n-2}} & b_2U^{(-t)}_{g_{n-2}} \cr b_3U^{(-t)}_{g_{n-2}} & T^{(-t)}_{g_{n-2}}+b_4U^{(-t)}_{g_{n-2}} \end{pmatrix}. $$
So from 
$$\begin{pmatrix} T_{g_n} & U_{g_n} \cr T_{g_{n}+1} & U_{g_{n}+1}  \end{pmatrix}=M^{g_n}=M^{sg_{n-1}-tg_{n-2}}=M^{sg_{n-1}}M^{-tg_{n-2}}\ , $$
we easily find the explicit expressions of $U_{g_n}$ and $T_{g_n}$.
Finally, to complete the proof, we only need to divide both $U_{g_n}$ and $T_{g_n}$ by the product $T_{g_{n-1}}^{(s)}T_{g_{n-2}}^{(-t)}$ and consider their resulting ratio multiplied by $-q$
$$-q\cfrac{U_{g_n}}{T_{g_n}}=-q\cfrac{a_2\cfrac{U_{g_{n-1}}^{(s)}}{T_{g_{n-1}}^{(s)}}+b_2\cfrac{U_{g_{n-2}}^{(-t)}}{T_{g_{n-2}}^{(-t)}}+(a_1b_2+a_2b_4)\cfrac{U_{g_{n-1}}^{(s)}}{T_{g_{n-1}}^{(s)}}\cfrac{U_{g_{n-2}}^{(-t)}}{T_{g_{n-2}}^{(-t)}}}{1+a_1\cfrac{U_{g_{n-1}}^{(s)}}{T_{g_{n-1}}^{(s)}}+b_1\cfrac{U_{g_{n-2}}^{(-t)}}{T_{g_{n-2}}^{(-t)}}+(a_1b_1+a_2b_3)\cfrac{U_{g_{n-1}}^{(s)}}{T_{g_{n-1}}^{(s)}}\cfrac{U_{g_{n-2}}^{(-t)}}{T_{g_{n-2}}^{(-t)}}}\ ,$$
rearranging the terms we obtain the equality (\ref{x}).
\end{proof}
\begin{corollary}
When $g=\mathcal{W}(2,3,1,-1)$, i.e., $g_{n-3}=F_n$, $\forall n\geq 3$, where $F=\mathcal{W}(0, 1, 1, -1)$ is the Fibonacci sequence, then for all $n \geq 5$
\begin{eqnarray*}
U_{F_n}&=&U_{F_{n-1}}T_{F_{n-2}}+T_{F_{n-1}}U_{F_{n-2}}+pU_{F_{n-1}}U_{F_{n-2}}=\\
&=&-q\left(U_{F_{n-1}}U_{F_{n-2}-1}+U_{F_{n-1}-1}U_{F_{n-2}}\right)+pU_{F_{n-1}}U_{F_{n-2}}\ ,\end{eqnarray*}
$$ x_{F_{n}}=\cfrac{qx_{F_{n-1}}+qx_{F_{n-2}}-px_{F_{n-1}}x_{F_{n-2}}}{q-x_{F_{n-1}}x_{F_{n-2}}}\ . $$
\end{corollary}
\begin{proof}
The identities immediately follow from Theorem \ref{teoUT}, observing that in this case we have $s=1$, $t=-1$, consequently $M^s=M=M^{-t}$, $U^{(s)}=U=U^{(-t)}$ and $T^{(s)}=T=T^{(-t)}$. 
\end{proof}
\begin{remark}
When $U$ is the Fibonacci sequence $F$, the previous Corollary provides a direct proof of the interesting formula
$$F_{F_{n}}=F_{F_{n-1}}F_{F_{n-2}-1}+F_{F_{n-1}-1}F_{F_{n-2}}+F_{F_{n-1}}F_{F_{n-2}}\ , $$
for all $n\geq 5$, avoiding many algebraic manipulations based on the basic properties of Fibonacci numbers. This formula is also true for $n=3$ and $n=4$, as a little calculation can show.
\end{remark}
The acceleration shown in the previous corollary is the same found by McCabe and Phillips in \cite{Aitken} with the secant method applied to the sequence $(x_n)_{n=2}^{+\infty}$ shifted by one step.\\
If we consider 
$$ M^{2n}=M^nM^n=\begin{pmatrix} T_n^2-qU_n^2 & 2T_nU_n+pU_n^2 \cr -q\left(2T_nU_n+pU_n^2\right) & T_n^2-qU_n^2+p\left(2T_nU_n+pU_n^2\right)  \end{pmatrix} \ ,$$
we find $ U_{2n}=2T_nU_n+pU_n^2$,  $ T_{2n}=T_n^2-qU_n^2\ , $ and
\begin{equation} \label{NewtonAcc} x_{2n}=-q\cfrac{U_{2n}}{T_{2n}}=\cfrac{2qx_n-px_n^2}{q-x_n^2}\ . \end{equation}
If we start from $x_2$, a repeated use of equation (\ref{NewtonAcc}) yields to the subsequence $x_2,x_4,x_8,...,x_{2^n},...$, corresponding to the subsequence obtained in \cite{Aitken}, using the Newton method on the sequence $(x_n)_{n=2}^{+\infty}$ shifted by one step.
\begin{proposition}
Using the notation of Theorem \ref{teoUT}, for $g=\mathcal{W}(h,h+k,2,1)$, i.e., $g_n=kn+h$  $\forall n\geq0$, then
for all $n\geq 2$ 
$$U_{g_n}=\cfrac{1}{q^{g_{n-2}}}\left(qU^2_{g_{n-1}}U_{g_{n-2}}+2T_{g_{n-1}}U_{g_{n-1}}T_{g_{n-2}}+pU^2_{g_{n-1}}T_{g_{n-2}}-U_{g_{n-2}}T_{g_{n-1}}^2\right) $$
$$ T_{g_n}=\cfrac{1}{q^{g_{n-2}}}\left(T_{g_{n-1}}^2T_{g_{n-2}}+pT_{g_{n-1}}^2U_{g_{n-2}}-qT_{g_{n-2}}U_{g_{n-1}}^2+2qT_{g_{n-1}}U_{g_{n-1}}U_{g_{n-2}}\right)$$
and
$$ x_{g_n}=\cfrac{x^2_{g_{n-1}}x_{g_{n-2}}+2qx_{g_{n-1}}-px^2_{g_{n-1}}-qx_{g_{n-2}}}{q-px_{g_{n-2}}-x^2_{g_{n-1}}+2x_{g_{n-1}}x_{g_{n-2}}}\ .  $$
\end{proposition}
\begin{proof}
We need to evaluate
$$M^{g_n}=M^{2g_{n-1}}M^{-g_{n-2}}.$$
Considering that
$$ M^{-n}=\cfrac{1}{q^n}\begin{pmatrix} p & -1 \cr q & 0  \end{pmatrix}^n=\cfrac{1}{q^n}\begin{pmatrix} T_n+pU_n & -U_n \cr qU_n & T_n  \end{pmatrix}, $$
we have
$$ M^{g_n}=\cfrac{1}{q^{g_{n-2}}}\begin{pmatrix} T^2_{g_{n-1}}-qU^2_{g_{n-1}} & 2T_{g_{n-1}}U_{g_{n-1}}+pU^2_{g_{n-1}} \cr  ... & ... \end{pmatrix}\begin{pmatrix} T_{g_{n-2}}+pU_{g_{n-2}} & -U_{g_{n-2}}  \cr qU_{g_{n-2}} & T_{g_{n-2}} \end{pmatrix}\ , $$
and
$$ U_{g_n}=\cfrac{1}{q^{g_{n-2}}}\left(qU^2_{g_{n-1}}U_{g_{n-2}}+2T_{g_{n-1}}U_{g_{n-1}}T_{g_{n-2}}+pU^2_{g_{n-1}}T_{g_{n-2}}-U_{g_{n-2}}T_{g_{n-1}}^2\right)\ , $$
$$ T_{g_n}=\cfrac{1}{q^{g_{n-2}}}\left(T_{g_{n-1}}^2T_{g_{n-2}}+pT_{g_{n-1}}^2U_{g_{n-2}}-qT_{g_{n-2}}U_{g_{n-1}}^2+2qT_{g_{n-1}}U_{g_{n-1}}U_{g_{n-2}}\right).  $$
Finally, dividing both $U_{g_n}$ and $T_{g_n}$ by  $T_{g_{n-1}}^2T_{g_{n-2}}$, their ratio becomes
$$\cfrac{U_{g_n}}{T_{g_n}}=\cfrac{q\left(\cfrac{U_{g_{n-1}}}{T_{g_{n-1}}}\right)^2\cfrac{U_{g_{n-2}}}{T_{g_{n-2}}}+2\cfrac{U_{g_{n-1}}}{T_{g_{n-1}}}+p\left(\cfrac{U_{g_{n-1}}}{T_{g_{n-1}}}\right)^2-\cfrac{U_{g_{n-2}}}{T_{g_{n-2}}}}{1+p\cfrac{U_{g_{n-2}}}{T_{g_{n-2}}}-q\left(\cfrac{U_{g_{n-1}}}{T_{g_{n-1}}}\right)^2+2q\cfrac{U_{g_{n-1}}}{T_{g_{n-1}}}\cfrac{U_{g_{n-2}}}{T_{g_{n-2}}}} $$
from which, remembering that $x_{g_n}=-q\cfrac{U_{g_n}}{T_{g_n}}$, with simple calculations we obtain the thesis.
\end{proof}
\begin{remark}
Applying the last Proposition to the Fibonacci numbers, we get another example of a beautiful identity, easily proved without hard calculations. In fact,  when we consider $g_n=kn$, we have for all $n\geq 2$
$$F_{kn}=(-1)^{k(n-2)}(-F_{k(n-1)}^2F_{k(n-2)}+2F_{k(n-1)}F_{k(n-1)-1}F_{k(n-2)-1}$$
$$+F^2_{k(n-1)}F_{k(n-2)-1}-F_{k(n-2)}F^2_{k(n-1)-1}),$$
which in particular , for $k=1$, becomes
$$F_n=(-1)^{n}\left(-F_{n-1}^2F_{n-2}+2F_{n-1}F_{n-2}F_{n-3}+F^2_{n-1}F_{n-3}-F_{n-2}^3\right).$$
\end{remark}
\section{Approximation methods}
In the previous section we have already seen how some accelerations are connected to some approximation methods, like  Newton and secant methods. In this section, as particular cases of our work, we  exactly retrieve the accelerations provided by these methods and studied in the papers \cite{Gill} and \cite{Aitken}. Furthermore, using our approach, we can show similar results for different approximation methods, like the Halley method. Finally, we will examine the acceleration provided by Householder method, which is a generalization of Newton and Halley methods. First of all, we recall that the secant method, applied to the equation $f(t)=0$, provides rational approximations of a root, starting from two approximations $x_0$ and $x_1$, through the recurrence relation
$$x_n=x_{n-1}-\cfrac{f(x_{n-1})(x_{n-1}-x_{n-2})}{f(x_{n-1})-f(x_{n-2})}\ ,\quad \forall n\geq2,$$
which becomes, when $f(t)=at^2-bt-c$
$$x_n=\cfrac{ax_{n-1}x_{n-2}-c}{ax_{n-1}+ax_{n-2}-b}.$$  

\begin{thm} \label{secant}
Let us consider the Fibonacci sequence $F$ and the sequences $U=\mathcal{W}(0,1,p,q)$, $T=\mathcal{W}(1,0,p,q)$, and $x=(x_n)_{n=2}^{+\infty}$, defined by equation (\ref{eq:x_n}). The subsequence $(x_{F_{n+2}+1})_{n=0}^{+\infty}$ corresponds to the approximations of the root of larger modulus of $t^2-pt+q$ provided by secant method.
\end{thm}
\begin{proof}
We take the sequence $g$ defined by $g_n=F_{n+2}+1$ for every $n\geq0$. Since $g_n-1=F_{n+2}$, we have the recurrent formula $g_n=g_{n-1}+g_{n-2}-1$. Let $M$ be the companion matrix of the sequence $U$. From the matrix product
$$M^{g_n}=M^{g_{n-1}}M^{g_{n-2}}M^{-1}$$
we immediately observe that
$$U_{g_n}=U_{g_{n-1}}U_{g_{n-2}}-\cfrac{1}{q}T_{g_{n-1}}T_{g_{n-2}}, \quad T_{g_n}=U_{g_{n-2}}T_{g_{n-1}}+U_{g_{n-1}}T_{g_{n-2}}+\cfrac{p}{q}T_{g_{n-1}}T_{g_{n-2}}\ .$$
Thus
$$\cfrac{U_{g_n}}{T_{g_n}}=\cfrac{\cfrac{U_{g_{n-1}}U_{g_{n-2}}}{T_{g_{n-1}}T_{g_{n-2}}}-\cfrac{1}{q}}{\cfrac{U_{g_{n-2}}}{T_{g_{n-2}}}+\cfrac{U_{g_{n-1}}}{T_{g_{n-1}}}+\cfrac{p}{q}}\ ,$$
and
$$x_{g_n}=-q\cfrac{U_{g_n}}{T_{g_n}}=\cfrac{x_{g_{n-1}}x_{g_{n-2}}-q}{x_{g_{n-2}}+x_{g_{n-1}}-p}\ ,$$
clearly proving the thesis. 
\end{proof}
The Newton method provides rational approximations for a solution of the equation $f(t)=0$. When $f(t)=at^2-bt-c$, given an initial approximation $y_0$, we recall that the Newton approximations sequence satisfies the recurrence relation
\begin{equation}\label{new}y_n=y_{n-1}-\cfrac{ay^2_{n-1}-by_{n-1}-c}{2ay_{n-1}-b}=\cfrac{ay^2_{n-1}+c}{2ay_{n-1}-b},\quad \forall n\geq1 \ . \end{equation}
On the other hand, the Halley method generates approximations for a solution of $t^2-pt+q=0$ through  the following recurrence relation
\begin{equation}\label{hal}y_n=y_{n-1}+\cfrac{(y_{n-1}^2-py_{n-1}+q)(p-2y_{n-1})}{3y^2_{n-1}-3py_{n-1}+p^2-q},\quad\forall n\geq1 \ ,\end{equation}
for a given initial step $y_0$.
In the next theorem we will show how particular subsequences of the sequence $x$, defined by equation (\ref{eq:x_n}), generate the Newton and Halley approximations to the root of larger modulus of $t^2-pt+q\ .$
\begin{thm} \label{Newtonteo}
Given the sequences $U=\mathcal{W}(0,1,p,q)$, $T=\mathcal{W}(1,0,p,q)$, and $x$, defined by equation (\ref{eq:x_n}),
then the recurrence relations
\begin{equation}\label{newrec}
 x_{2n-1}=\cfrac{x_n^2-q}{2x_n-p},\end{equation}
 \begin{equation}\label{halrec}
 x_{3n-2}=x_{n-1}+\cfrac{(x_{n-1}^2-px_{n-1}+q)(p-2x_{n-1})}{3x^2_{n-1}-3px_{n-1}+p^2-q}\ ,
 \end{equation}
respectively generates the subsequences of $x$ producing Newton and Halley approximations for the root of larger modulus of $t^2-pt+q \ .$
\end{thm}
\begin{proof}
The proof of relation (\ref{newrec}) is a direct consequence of Corollary \ref{cor:xshift} 
$$
x_{2n-1}=x_{n+n-1}=\cfrac{x_n^2-q}{2x_n-p}\ .
$$
Starting from $x_2$, a repeated use of this equation yields the subsequence $$\left(x_2,x_3,x_5,\ldots,x_{2^n+1},\ldots\right)\ ,$$ which consists of the Newton approximations for the root of larger modulus of $t^2-pt+q\ .$
Relation (\ref{newrec}) and  Corollary \ref{cor:xshift} also provide the proof of recurrence (\ref{halrec}) for Halley approximations 
\begin{eqnarray*}
x_{3n-2}&=&x_{n+2n-2}=\cfrac{x_{2n-1}x_n-q}{x_n+x_{2n-1}-p}=\\
 &=&\cfrac{\cfrac{x_n^2-q}{2x_n-p}x_n-q}{x_n+\cfrac{x_n^2-q}{2x_n-p}-p}=\cfrac{x_{n}^3-3qx_{n}+pq}{3x^2_{n}-3px_{n}+p^2-q}=\\
&=&x_{n-1}+\cfrac{(x_{n-1}^2-px_{n-1}+q)(p-2x_{n-1})}{3x^2_{n-1}-3px_{n-1}+p^2-q}\ .
\end{eqnarray*}
\end{proof}
To complete our overview on approximation methods and generalize the results of previous theorem, we consider the Householder method \cite{House}, which provides approximations for a solution of a nonlinear equation $f(t)=0$. The Householder recurrence relation is 
\begin{equation} \label{house-method}  y_{n+1}=y_n+d\cfrac{(1/f)^{(d-1)}(y_n)}{(1/f)^{(d)}(y_n)},\quad \forall n\geq0  \end{equation}
for a given initial approximation $y_0$, where $f^{(d)}$ is the $d$--th derivative of $f$. The Newton and Halley method clearly are particular cases of this method, for $d=1$ and $d=2$ respectively. Indeed, when $f(t)=t^2-pt+q$ from (\ref{house-method}), when $d=1$ or $d=2$, it is easy to retrieve the formulas (\ref{new}) and (\ref{hal}).
\begin{thm}
With the same hypotheses of Theorem \ref{Newtonteo} , if $y_n=x_k$ for some $k\geq2$, then $y_{n+1}=x_{(d+1)k-d}$, where $y_{n+1}$ is obtained from relation (\ref{house-method}) when $f(t)=t^2-pt+q$.
\end{thm}
\begin{proof}
Let $\alpha_1$, $\alpha_2$ be the roots of $t^2-pt+q$, we can write
\small
$$\left(\cfrac{1}{t^2-pt+q}\right)^{(d)}=\left( \cfrac{ (t-\alpha_1)^{-1}-(t-\alpha_2)^{-1}}{\alpha_1-\alpha_2}\right)^{(d)}
=\cfrac{(-1)^dd!\left((t-\alpha_2)^{d+1}-(t-\alpha_1)^{d+1}\right)}{(\alpha_1-\alpha_2)(t-\alpha_1)^{d+1}(t-\alpha_2)^{d+1}}\ .$$
\normalsize
Thus, relation (\ref{house-method}) becomes
$$y_{n+1}=y_n-\cfrac{(y_n-\alpha_1)(y_n-\alpha_2)\left((y_n-\alpha_2)^{d}-(y_n-\alpha_1)^{d}\right)}{(y_n-\alpha_2)^{d+1}-(y_n-\alpha_1)^{d+1}}\ ,$$
and with some algebraic manipulations we obtain
$$y_{n+1}=\cfrac{\alpha_1(y_n-\alpha_2)^{d+1}-\alpha_2(y_n-\alpha_1)^{d+1}}{(y_n-\alpha_2)^{d+1}-(y_n-\alpha_1)^{d+1}}\ .$$
If we set $y_n=x_k=\cfrac{U_k}{U_{k-1}}=\cfrac{\alpha_1^k-\alpha_2^k}{\alpha_1^{k-1}-\alpha_2^{k-1}}$, then
\begin{eqnarray*}y_{n+1}&=&\cfrac{\alpha_1\left(\cfrac{\alpha_1^k-\alpha_2\alpha_1^{k-1}}{\alpha_1^{k-1}-\alpha_2^{k-1}}\right)^{d+1}-\alpha_2\left(\cfrac{\alpha_1\alpha_2^{k-1}-\alpha_2^{k}}{\alpha_1^{k-1}-\alpha_2^{k-1}}\right)^{d+1}}{\left(\cfrac{\alpha_1^k-\alpha_2\alpha_1^{k-1}}{\alpha_1^{k-1}-\alpha_2^{k-1}}\right)^{d+1}-\left(\cfrac{\alpha_1\alpha_2^{k-1}-\alpha_2^{k}}{\alpha_1^{k-1}-\alpha_2^{k-1}}\right)^{d+1}}=\\
&=&\cfrac{\alpha_1^{(d+1)k-d}(\alpha_1-\alpha_2)^{d+1}-\alpha_2^{(d+1)k-d}(\alpha_1-\alpha_2)^{d+1}}{\alpha_1^{(d+1)k-d-1}(\alpha_1-\alpha_2)^{d+1}-\alpha_2^{(d+1)k-d-1}(\alpha_1-\alpha_2)^{d+1}}=\\
&=& \cfrac{U_{(d+1)k-d}}{U_{(d+1)k-d-1}}=x_{(d+1)k-d}\ .
\end{eqnarray*}
\end{proof}

\section{Applications to continued fractions}
In order to show some interesting applications of the exposed results, we study the acceleration of a sequence of rationals which come from the sequence of convergents of a certain continued fraction. The continued fraction, which we will introduce, provides a periodic representation of period 2 for every quadratic irrationality (except for the square roots). Surely its convergents do not generate the best approximations for these irrationalities (because we use rational partial quotients, instead of integers), but between them, using the acceleration method of the previous section, we can find at the same time the approximations derived from the the secant method, the Newton method and the Halley method.\\
We remember that a continued fraction is a representation of a real number $\alpha$ through a sequence of integers as follows:
$$\alpha=a_0+\cfrac{1}{a_1+\cfrac{1}{a_2+\cfrac{1}{a_3+\cdots}}}\ ,$$
where the integers $a_0,a_1,\ldots$ can be evaluated with the recurrence relations 
$$\begin{cases} a_k=[\alpha_k]\cr \alpha_{k+1}=\cfrac{1}{\alpha_k-a_k} \quad \text{if} \ \alpha_k \ \text{is not an integer}  \end{cases}  \quad k=0,1,2,\ldots \ ,$$
for $\alpha_0=\alpha$ (see, e.g., \cite{Olds}). A continued fraction can be expressed in a compact way using the notation $[a_0,a_1,a_2,a_3,\ldots]$. The finite continued fraction
$$[a_0,\ldots,a_n]=C_n=\cfrac{p_n}{q_n}\ ,\quad n=0,1,2,\ldots \ ,$$
is a rational number and is called the $n$--th \emph{convergent} of $[a_0,a_1,a_2,a_3,\ldots]$ and the $a_i$'s are called \emph{partial quotients}. Furthermore, the sequences $(p_n)_{n=0}^{+\infty}$ and $(q_n)_{n=0}^{+\infty}$ are recusively defined by the equations
\begin{equation} \label{conv} \begin{cases} p_n=a_np_{n-1}+p_{n-2} \cr q_n=a_nq_{n-1}+q_{n-2}  \end{cases},\quad \forall n\geq2 \ , \end{equation}
and initial conditions $p_0=a_0$, $p_1=a_0a_1+1$, $q_0=1$, $q_1=a_1\ .$
In the following we will use rational numbers as partial quotients, instead of the usual integers. 
\begin{remark}
In general a continued fraction can have complex numbers as partial quotients. In this case given a real number there are no algorithms which provide the partial quotients. However, such a continued fraction can converge to a real number. In the following, we will study continued fractions of period 2 with rational partial quotients which are convergents to the root of larger modulus of a quadratic equation. For a deep study of the convergence of continued fractions, i.e., for an analytic theory of these objects, see \cite{Wall}.
\end{remark}
\begin{remark}
A continued fraction with rational partial quotients $\left[\cfrac{a_0}{b_0},\cfrac{a_1}{b_1},...\right]$ has an equivalent form as
\small $$ \cfrac{a_0}{b_0}+\cfrac{b_1}{a_1+\cfrac{b_1b_2}{a_2+\cfrac{b_2b_3}{a_3+\ddots}}}, $$\normalsize
but the representation
\small $$  \cfrac{a_0}{b_0}+\cfrac{1}{\cfrac{a_1}{b_1}+\cfrac{1}{\cfrac{a_2}{b_2}+\cfrac{1}{\cfrac{a_3}{b_3}+\ddots}}} $$ \normalsize
is more suitable for the study of the convergents, as we will see soon.
\end{remark}
When we study the continued fraction $\left[\cfrac{a_0}{b_0},\cfrac{a_1}{b_1},\cfrac{a_2}{b_2},...\right]$, the equations (\ref{conv}) become
\begin{equation} \label{conv1} \begin{cases} p_n=\cfrac{a_n}{b_n} p_{n-1}+p_{n-2} \cr q_n=\cfrac{a_n}{b_n} q_{n-1}+q_{n-2}  \end{cases},\quad \forall n\geq2 \ , \end{equation}
with initial conditions $p_0=\cfrac{a_0}{b_0}$, $p_1=\cfrac{a_0}{b_0}\cfrac{a_1}{b_1}+1$, $q_0=1$, $q_1=\cfrac{a_1}{b_1}\ ,$ 
providing two sequences of rationals . In the next Proposition we study how to determine the convergents through the ratio of two recurrent sequences of integers.
\begin{proposition} \label{conv-prop}
Given the continued fraction $\left[\cfrac{a_0}{b_0},\cfrac{a_1}{b_1},\cfrac{a_2}{b_2},...\right]$, let $(p_n)_{n=0}^{+\infty}$ and $(q_n)_{n=0}^{+\infty}$ be the sequences which provide the sequence of convergents $(C_n)_{n=0}^{+\infty}$ defined recursively by equations (\ref{conv1}). If we consider the sequences $(s_n)_{n=0}^{+\infty}$, $(t_n)_{n=0}^{+\infty}$, $(u_n)_{n=0}^{+\infty}$ defined for all $n\geq 2$
$$  \begin{cases}s_n=a_ns_{n-1}+b_nb_{n-1}s_{n-2}\cr t_n=a_nt_{n-1}+b_nb_{n-1}t_{n-2}\cr u_n=b_nu_{n-1} \end{cases}\text{and initial conditions} \quad \begin{cases} s_0=a_0, s_1=a_0a_1+b_0b_1 \cr t_0=1, t_1=a_1 \cr u_0=1\end{cases}$$
then $p_n=\cfrac{s_n}{b_0u_n}$ and $q_n=\cfrac{t_n}{u_n}$, for every $n\geq 0$, i.e.,
$$C_n=\cfrac{s_n}{b_0t_n},\quad \forall n\geq0.$$
\end{proposition}
\begin{proof}
We prove the thesis by induction. When $n=0$, and $n=1$ we have respectively
$$p_0=\cfrac{a_0}{b_0}=\cfrac{s_0}{b_0u_0},\quad q_0=1=\cfrac{t_0}{u_0}$$
$$p_1=\cfrac{a_0a_1+b_0b_1}{b_0b_1}=\cfrac{s_1}{b_0u_1},\quad q_1=\cfrac{a_1}{b_1}=\cfrac{t_1}{u_1}\ .$$
Now considering that
$$p_n=\cfrac{a_n}{b_n}p_{n-1}+p_{n-2},\quad n\geq2\ ,$$
by induction hypothesis we have
\begin{eqnarray*}
p_n&=&\cfrac{a_n}{b_n}\cdot\cfrac{s_{n-1}}{b_0u_{n-1}}+\cfrac{s_{n-2}}{b_0u_{n-2}}=\cfrac{a_ns_{n-1}u_{n-2}+b_nu_{n-1}s_{n-2}}{b_0b_nu_{n-1}u_{n-2}}=\\
&=&\cfrac{a_ns_{n-1}u_{n-2}+b_nb_{n-1}u_{n-2}s_{n-2}}{b_0u_nu_{n-2}}=\cfrac{s_n}{b_0u_n}\ . \end{eqnarray*}
Similarly we obtain $q_n=\cfrac{t_n}{u_n}$.
\end{proof}
We finally focus our attention on a particular continued fraction of period 2. Precisely, we consider the continued fraction
\begin{equation} \label{fc} \left[\ \overline{\cfrac{b}{a},\cfrac{b}{c}}\ \right] \end{equation}
for $a,b,c\in\mathbb Z$ not zero. This continued fraction converges to the quadratic irrationality, of larger modulus, root of $ax^2-bx-c$. Indeed, it is easy to check that
$$\alpha=\cfrac{b}{a}+\cfrac{1}{\cfrac{b}{c}+\cfrac{1}{\alpha}}=\cfrac{b}{a}+\cfrac{c\alpha}{b\alpha+c}\Leftrightarrow ab\alpha^2+ac\alpha=b^2\alpha+bc+ac\alpha\Leftrightarrow a\alpha^2-b\alpha-c=0.$$
From the previous Proposition we know that the convergents of (\ref{fc}) are determined by a linear recurrent sequence. Indeed, in this case the sequences $s=(s_n)_{n=0}^{+\infty}$ and $(t_n)_{n=0}^{+\infty}$ of the Proposition \ref{conv-prop} correspond to
$$s=\mathcal{W}(b,b^2+ac,b,-ac),\quad t=\mathcal W(1,b,b,-ac)\ .$$
Thus, defining the sequence
$$\sigma=\mathcal W(0,1,b,-ac)\ ,$$
we obtain that the convergents $(C_n)$ of our continued fraction are
$$C_n=\cfrac{\sigma_{n+2}}{a\sigma_{n+1}},\quad \forall n\geq0\ .$$
Clearly, they are not the best approximations of the quadratic irrationality, but since the convergents are determined by the ratio of the sequence $\sigma$, we can apply our method of acceleration to such a sequence. In particular, we will see that properly accelerating the convergents sequence, we can find the approximations derived by  Newton, Halley and secant method, which compare at the same time between these approximations.
\begin{thm}
Let us consider the continued fraction $\left[\ \overline{\cfrac{b}{a},\cfrac{b}{c}}\ \right]$, whose sequence of convergents is  $(C_n)_{n=0}^{+\infty}$, and the sequence $\sigma=\mathcal{W}(0,1,b,-ac)$. If $\alpha$ is the root of larger modulus of $at^2-bt-c$, then 
\begin{enumerate}
\item $(C_{F_{n+2}-1})_{n=0}^{+\infty}$, are the approximations of $\alpha$  by the secant method;
\item $(C_{2^n-1})_{n=0}^{+\infty}$, are the Newton approximations of $\alpha$;
\item $(C_{3^n-1})_{n=0}^{+\infty}$, are the Halley approximations of $\alpha$.
\end{enumerate}
\end{thm}
\begin{proof}
We use the auxiliary sequence $x=(x_n)_{n=0}^{+\infty}$ defined by
$$x_n=\cfrac{\sigma_n}{\sigma_{n-1}},\quad \forall n\geq2.$$
We can directly apply Theorem \ref{secant}, obtaining that $(x_{F_{n+2}+1})_{n=0}^{+\infty}$ are the approximations through the secant method of the root of larger modulus of $t^2-bt-ac$. Considering the sequence $\bar x$ defined by
$$\bar x_n=\cfrac{x_n}{a},\quad \forall n\geq2,$$
we have
$$\bar x_{g_n}=\cfrac{x_{g_n}}{a}=\cfrac{1}{a}\cdot\cfrac{x_{g_{n-1}}x_{g_{n-2}}+ac}{x_{g_{n-2}}+x_{g_{n-1}}-b}=\cfrac{a\bar x_{g_{n-1}}\bar x_{g_{n-2}}+c}{a\bar x_{g_{n-2}}+a\bar x_{g_{n-1}}-b},$$
i.e. $(\bar x_{g_n})_{n=0}^{+\infty}$ are the approximations through the secant method of $\alpha$. Thus, remembering that $aC_n=x_{n+2}$ for every $n\geq0$, we have that $(C_{F_{n+2}-1})_{n=0}^{+\infty}$ are the approximations through the secant method of $\alpha$.\\
Applying Theorem \ref{Newtonteo} to the sequence $x$ we obtain 
$$x_{2n-1}=\cfrac{x_n^2+ac}{2x_n-b}\ ,$$
i.e., $x_2$, $x_3$, $x_5$, $x_9$, $\ldots$, $x_{2^n+1}$, $\ldots\,$ are the Newton approximations for the root of larger modulus of $t^2-bt-ac$. Now it is easy to see that $\bar x_2,\bar x_3,\bar x_5,...,\bar x_{2^n+1}$ are the Newton approximations of $\alpha$ for every $n\geq2$, i.e., $(C_{2^n-1})_{n=0}^{+\infty}$ are the Newton approximations of $\alpha$.\\
Finally, using Theorem \ref{Newtonteo} again, with similar observations, it is possible to prove that $(C_{3^n-1})_{n=0}^{+\infty}$ are the Halley approximations of $\alpha$.
\end{proof}

\end{document}